\documentclass[11pt]{amsart}
\usepackage{color}

\newtheorem{theorem}{Theorem}[section]
\newtheorem{lemma}[theorem]{Lemma}

\theoremstyle{definition}

\newtheorem{proposition}[theorem]{Proposition}
\newtheorem{corollary}[theorem]{Corollary}

\theoremstyle{remark}

\numberwithin{equation}{section}

\def\R{\mathbb{R}}
\def\calC{\mathcal{C}}
\def\blambda{\bar{\lambda}}
\def\tlambda{\tilde{\lambda}}

\begin{document}

\title[A MEMS model with fringing filed]{Radial regular and rupture  solutions for a MEMS model with fringing field}

\author{Marius Ghergu}
\thanks{}
\address{School of Mathematics and Statistics, University College Dublin, Belfield, Dublin 4, Ireland}
\address{Institute of Mathematics Simion Stoilow of the Romanian Academy, 21 Calea Grivitei St., 010702 Bucharest, Romania}
\email{marius.ghergu@ucd.ie}

\author{Yasuhito Miyamoto}
\thanks{The second author was supported by JSPS KAKENHI Grant Numbers 19H01797 and 19H05599.}
\address{Graduate School of Mathematical Sciences, The University of Tokyo, 3-8-1 Komaba, Meguro-ku, Tokyo 153-8914, Japan}
\email{miyamoto@ms.u-tokyo.ac.jp}




\begin{abstract}
We investigate radial solutions for the problem 
\[
\begin{cases}
\displaystyle -\Delta U=\frac{\lambda+\delta|\nabla U|^2}{1-U},\; U>0 & \textrm{in}\ B,\\
U=0 & \textrm{on}\ \partial B,
\end{cases}
\]
which is related to the study of Micro-Electromechanical Systems (MEMS). Here, $B\subset \R^N$ $(N\geq 2)$ denotes the open unit ball and $\lambda, \delta>0$ are real numbers. Two classes of solutions are considered in this work: (i) {\it regular solutions}, which satisfy $0<U<1$ in $B$ and (ii) {\it rupture solutions} which satisfy $U(0)=1$, and thus make the equation singular at the origin. 
Bifurcation with respect to parameter $\lambda>0$ is also discussed.

\end{abstract}

\maketitle 

\section{Introduction and the main results}
In this paper we are concerned with the problem
\begin{equation}\label{E2}
\begin{cases}
\displaystyle -\Delta U=\frac{\lambda+\delta|\nabla U|^2}{1-U},\; U>0 & \textrm{in}\ B,\\
U=0 & \textrm{on}\ \partial B,
\end{cases}
\end{equation}
where $B\subset \R^N$ ($N\geq 2$) denotes the open unit ball and $\lambda, \delta>0$ are real numbers.
The study of \eqref{E2} is motivated by the more general problem
\begin{equation}\label{E2v1}
\begin{cases}
\displaystyle -\Delta U=\frac{\lambda+\delta|\nabla U|^2}{(1-U)^p},\; U>0 & \textrm{in}\ B,\\
U=0 & \textrm{on}\ \partial B,
\end{cases}
\end{equation}
where $p\geq 1$. The case $p=2$ in \eqref{E2v1} was discussed in \cite{WY10}. As the authors in \cite{WY10} emphasized, their aproach is suitable to treat \eqref{E2v1} for all $p>1$.
The problem \cite{WY10} with $p=2$  arises in the mathematical modelling of  the Micro-Electromechanical Systems (MEMS). In such a context, $\lambda$ represents the applied voltage while $\delta |\nabla u|^2$ is related to the effect of a fringing electrostatic field (see, e.g., \cite{LW08}).

In this paper we are interested in the study of two classes of radial solutions to \eqref{E2} namely
\begin{itemize}
\item {\it regular solutions}, which satisfy $0<U<1$ in $B$.
\item {\it rupture solutions} which satisfy $U(0)=1$, and thus make the main equation in \eqref{E2} singular at the origin. 
\end{itemize}
The results for radial solutions to \eqref{E2v1}  obtained in \cite{WY10} show that:
\begin{itemize}
\item There exists $\Lambda>0$ such that:
\begin{enumerate}
\item[(i)] problem \eqref{E2v1} has no regular solutions for $\lambda>\Lambda$.
\item[(ii)]  problem \eqref{E2v1} has a unique regular solution if $\lambda=\Lambda$.
\item[(iii)] problem \eqref{E2v1} has at least two regular solutions for $0<\lambda<\Lambda$.
\end{enumerate}
\item Problem \eqref{E2v1} has no rupture solutions for all $\lambda>0$.
\end{itemize}
In this paper, our results for \eqref{E2} reveal a striking difference to those for \eqref{E2v1}. More precisely, we show that:
\begin{itemize}
\item Problem \eqref{E2} has infinitely many regular solutions for suitable $\lambda$ and $\delta$ (see Theorem~\ref{thmain1} (ii) below).
\item If $0<\delta <N/2$ then problem \eqref{E2} has exactly one rupture solution (see Theorem~\ref{thmain1} below).
\item If $\delta \geq N/2$ then, for $\lambda>0$ small, problem \eqref{E2} has infinitely many rupture solutions (see Theorem~\ref{thmain3} below).
\end{itemize} 
Finally, let us mention that non-radial rupture solutions for the problem
\begin{equation}\label{E2v2}
\begin{cases}
\displaystyle -\Delta U= \frac{\lambda(1+|\nabla U|^2)}{(1-U)^2},\; U>0 & \textrm{in}\ \Omega,\\
U=0 & \textrm{on}\ \partial \Omega,
\end{cases}
\end{equation}
where $\Omega\subset \R^2$ is a smooth domain, are discussed in \cite{DW12}. It is obtained in \cite{DW12} that  \eqref{E2v2} admits a solution which develops an isolated rupture as $\lambda\to 0$. Furthermore, if $\Omega$ is not simply connected, then, for any $m\geq 1$ problem \eqref{E2v2} admits a solution with $m$ isolated ruptures as $\lambda\to 0$.

Radial regular solutions of \eqref{E2} satisfy 
\begin{equation}\label{E4'}
\begin{cases}
U''+\frac{N-1}{r}U'+\frac{\lambda+\delta(U')^2}{1-U}=0 & \textrm{for}\ 0<r<1,\\
U(0)=\alpha\in (0,1),\\
U'(0)=0,\; U(1)=0.
\end{cases}
\end{equation}
In our approach we will first try to use a change of unknown in order to reduce the gradient term $|\nabla U|^2$ in \eqref{E2}. Thus, taking $U=1-\phi(u)$  in \eqref{E2} we see that $u$ fulfills  
\[
\begin{cases}
\displaystyle \Delta u=\frac{\lambda}{\phi(u)\phi'(u)},\; u>0 & \textrm{in}\ B,\\
u=a\geq 0 & \textrm{on}\ \partial B,
\end{cases}
\]
provided $\phi$ satisfies
\begin{equation}\label{phii}
\phi'(u)=\phi(u)^\delta\;\,\mbox{ and }\;\, \phi(a)=1.
\end{equation}
Solving \eqref{phii} we are led to three distinct shapes of $\phi$ according to the cases $0<\delta<1$, $\delta=1$ and $\delta>1$. We shall see that problem \eqref{E2} (and also \eqref{E4'}) is equivalent to (\ref{Ev1}), (\ref{Ev3}) or (\ref{Ev2}) which features a more convenient form to study. In particular, by the classical result in \cite{GNN79}, any regular solution $U$ of \eqref{E2} must be radial and thus, it satisfies \eqref{E4'} for some $\alpha>0$. Furthermore, for fixed $\delta>0$, the bifurcation diagram is given by the graph 
$\calC:=\{(\alpha, \lambda(\alpha)):\ 0<\alpha<1\}$,
and $\calC$ determines the solution structure. We easily see that $\lim_{\alpha\to 0}\lambda(\alpha)=0$.
In order to describe the shape of the graph $\lambda(\alpha)$ let us introduce the following definition.

\medskip

\noindent{\bf Definition. } {\it
Let $\calC$ be the bifurcation curve of (\ref{E4'}).
\begin{enumerate}
\item[(i)] We call $\calC$  of {\it Type I} if  there exists $\lambda^*>0$ such that $\lambda(\alpha)\to \lambda^*$ $(\alpha\to 1)$ and $\lambda(\alpha)$ is strictly increasing.
\item[(ii)] We call $\calC$ of {\it Type II} if there exists $\lambda^*>0$ such that $\lambda(\alpha)\to \lambda^*$ $(\alpha\to 1)$ and $\lambda(\alpha)$ oscillates around $\lambda^*$.

\end{enumerate}
}
We will see that for $0<\delta<N/2$ the above $\lambda^*$ is given by $\lambda^*=N-1-\delta$. Our main result concerning regular solutions of \eqref{E2} is stated below.

\begin{theorem}\label{thmain1} Assume $N\geq 2$ and $0<\delta<N/2$. The following hold true.
\begin{enumerate}
\item[(i)] If $N\ge 3$ and $\delta\le (N-2-2\sqrt{N-1})/2$, then the bifurcation diagram of \eqref{E2} is of Type I.
Furthermore, \eqref{E2} has a unique rupture solution $(\lambda^*,U^*)$ given by
\begin{equation}\label{rupture1}
(\lambda^*,U^*)=(N-1-\delta, 1-r)
\end{equation}
and \eqref{E2} has exactly one radial regular solution for each fixed $\lambda\in(0,\lambda^*)$. 

\item[(ii)] If $N\ge 3$ and $\delta>(N-2-2\sqrt{N-1})/2$ or if $N=2$, then the bifurcation diagram of \eqref{E2} is of Type II.
Furthermore, \eqref{E2} has a unique rupture solution $(\lambda^*,U^*)$ given by \eqref{rupture1} and \eqref{E2} with $\lambda=\lambda^*$ has infinitely many radial regular solutions.
In particular, this occurs for $3\le N\le 6$ (since $(N-2-2\sqrt{N-1})/2<0$).
\end{enumerate}
\end{theorem}
Also, from the above result we deduce that if $\delta =1$ and $N\ge 3$ then  the following hold:
\begin{itemize}
\item If $3\le N\le 9$, then the bifurcation curve is of Type II.
\item  If $N\ge 10$, then the bifurcation curve is of Type I.
\end{itemize}

Fixing $\delta>0$ we are next interested in the range of $\lambda>0$ for which problem \eqref{E2} has regular solutions. We obtain the following result. 

\begin{theorem}\label{thmain2}
Let $N\ge 2$.
\begin{enumerate}
\item[(i)] Assume $\delta>0$. If \eqref{E2} has a regular solution then
\[
\lambda<\min\left\{ \frac{\mu_1}{4}, \frac{\mu_1}{\delta}\right\},
\]
where $\mu_1>0$ is the first eigenvalue of the Dirichlet Laplacian on $B$.
\item[(ii)] Assume {$\delta\ge N/2$}.
Then, there exists $\bar{\lambda}>0$ such that \eqref{E2} has exactly two radial regular solutions for each $\lambda\in (0,\bar{\lambda})$, exactly one radial regular solution for $\lambda=\bar{\lambda}$ and no radial regular solution for $\lambda>\bar{\lambda}$.
Moreover,
\begin{equation}\label{barl}
\bar{\lambda}\ge N\left(\frac{2}{\delta+1}\right)^{(\delta+1)/(\delta-1)}.
\end{equation}
\item[(iii)] Assume $\delta=N/2$. Then all the radial regular solutions of \eqref{E2} can be described as
\begin{equation}\label{L9E0}
(\lambda,U(r))=(2N\alpha (1-\alpha),\alpha(1-r^2)),\ \ 0<\alpha<1.
\end{equation}
In particular, \eqref{E2} has exactly two radial regular solutions for each $\lambda\in (0,N/2)$, exactly one radial regular solution for $\lambda=N/2$, and no radial regular solution for $\lambda>N/2$.
\end{enumerate}
\end{theorem}

\noindent{\bf Remark. }
A parameter $\bar{\lambda}>0$ is called the extremal value for \eqref{E2} if 
\begin{itemize}
\item \eqref{E2} has a regular solution for $\lambda<\bar{\lambda}$.
\item \eqref{E2} has no regular solution for $\lambda>\bar{\lambda}$.
\end{itemize}
The solution for $\lambda=\bar{\lambda}$ is called the extremal solution of \eqref{E2}. By Theorem~\ref{thmain1} and Theorem~\ref{thmain2} we easily deduce the following.  

\begin{corollary} The following hold true.
\begin{enumerate}
\item[(i)] If $0<\delta\le(N-2-2\sqrt{N-1})/2$, then $\blambda=\lambda^*$ and the extremal solution is singular;
\item[(ii)]  If $(N-2-2\sqrt{N-1})/2<\delta<N/2$, then 
$$
\lambda^*<\blambda<\min\left\{\frac{\mu_1}{4},\frac{\mu_1}{\delta}\right\}
$$
and the extremal solution is regular;
\item[(iii)] If $\delta\ge N/2$, then 
$$
N\left(\frac{2}{\delta+1}\right)^{(\delta+1)/(\delta-1)}\le{\blambda}<\min\left\{\frac{\mu_1}{4},\frac{\mu_1}{\delta}\right\}.
$$
\end{enumerate}
\end{corollary}

\smallskip
We turn next to the study of rupture solutions to \eqref{E2}. From Theorem~\ref{thmain1} above we see that if $0<\delta<N/2$ then \eqref{E2} has a unique rupture solution given by \eqref{rupture1}. The result below presents our findings for $\delta\geq N/2$. 

\begin{theorem}\label{thmain3}
Assume $\delta\geq N/2$. The following hold true.
\begin{enumerate}
\item[(i)] If $N=2$, then there exists $\lambda^{**}>0$ such that  for each $\lambda\in (0,\lambda^{**})$, \eqref{E2} has infinitely many rupture solutions. Furthermore, if $\delta=1$ we have $\lambda^{**}=1$. 
\item[(ii)] If $N\geq 3$ and $N/2\leq \delta<N-1$ then, for any 
\begin{equation}\label{l3}
0<\lambda< \lambda^{***}:=\frac{\delta(N-1-\delta)}{\delta-1},
\end{equation}
problem \eqref{E2} has infinitely many rupture solutions $U(r)$ satisfying the following:
\begin{enumerate}
\item[(ii1)] For $N/2<\delta<N-1$, then
\begin{equation}\label{L7E-2}
U(r)=1-\sqrt{\frac{\lambda}{N-1-\delta}}r(1+o(1))\quad\textrm{as}\quad r\to 0.
\end{equation}
\item[(ii2)] For $\delta=N/2$, then there exists $c>1$ such that
\begin{equation}\label{L7E-1}
1-\sqrt{\frac{2\lambda}{N}}cr\le U(r)<1-\sqrt{\frac{2\lambda}{N}}r
\quad\textrm{for}\quad 0<r<1.
\end{equation}
\end{enumerate}
\item[(iii)] If $N\geq 3$ and $\delta\geq N-1$, then there exists $\lambda^{****}>0$ such that  for each $\lambda\in (0,\lambda^{****})$, problem \eqref{E2} has  infinitely many rupture solutions.

\end{enumerate}
\end{theorem}
We see that for $N/2\leq \delta<N-1$ we have
\begin{equation}\label{2lambda}
\lambda^*=N-1-\delta<\lambda^{***}=\delta(N-1-\delta)/(\delta-1)
\end{equation}
and that the rupture solution \eqref{rupture1} is included in both the cases (ii1) and (ii2) in Theorem~\ref{thmain3} above.

The results in Theorems~\ref{thmain1}-\ref{thmain3} can be sumarized in the tables below.
\begin{table}[h]
\begin{tabular}{|c|c|c|c|}
\hline
$\delta$ & $(0,1)$ & $1$ & $(1,\infty)$ \\
\hline
regular solutions & $(0,\lambda^*)$ & $(0,1]$ & $(0,\bar{\lambda}]$\rule[-2mm]{0mm}{6mm}\\
\hline
rupture solutions & $\lambda^*$ & (0,1) & $(0, \lambda^{**})$\rule[-2mm]{0mm}{6mm}\\
\hline
\end{tabular}
\vspace{0.3cm}
\caption{Case $N=2$; The range of $\lambda$ for which problem \eqref{E2} has a regular/rupture solution.}
\end{table}

\begin{table}[h]
\begin{tabular}{|c|c|c|c|c|}
\hline
$\delta $ & $(0,\frac{N}{2})$ & $\frac{N}{2}$ & $(\frac{N}{2},N-1)$
& $[N-1,\infty)$\rule[-2mm]{0mm}{6mm} \\
\hline
regular solutions&  $(0,\lambda^*)$ & $(0,\frac{N}{2}]$ & $(0,\bar{\lambda}]$ & $(0,\bar{\lambda}]$\rule[-2mm]{0mm}{6mm}\\
\hline
rupture solutions& $\lambda^*$ & $(0,\frac{N}{2})$ & {$(0,\lambda^{***})$} & {$(0,\lambda^{****})$}\rule[-2mm]{0mm}{6mm}\\
\hline
\end{tabular}
\vspace{0.3cm}
\caption{Case $N\geq 3$; The range of $\lambda$ for which problem \eqref{E2} has a regular/rupture solution.}
\end{table}

We see that in the study of regular solutions to \eqref{E2} we identify two critical parameters $\lambda^*$ and $\bar\lambda$ while in the study of rupture solutions we identify four critical parameters $\lambda^*$, $\lambda^{**}$, $\lambda^{***}$ and $\lambda^{****}$. Also $\lambda^*<\lambda^{***}$ (see \eqref{2lambda}).

\section{Some preliminary results}

In this section we collect some useful results in our approach to prove Theorems \ref{thmain1}-\ref{thmain3}.

\begin{proposition}[{see \cite{JL72, KS01}}]\label{PropJL}
Consider the problem
\begin{equation}\label{PropJLE0}
\begin{cases}
w''+\frac{N-1}{r}w'+\lambda (w+1)^{p}=0 & \textrm{for}\ 0<r<1,\\
w'(0)=0,\\
w(1)=0,
\end{cases}
\end{equation}
where $1<p\leq (N+2)/(N-2)$ if $N\ge 3$, and $1<p<\infty$ if $N=2$.
Then, there exists $\lambda_*>0$ such that:
\begin{itemize}
\item Problem \eqref{PropJLE0} has exactly two regular solutions for $0<\lambda<\lambda_*$;
\item Problem \eqref{PropJLE0} has exactly one regular solution for $\lambda=\lambda_*$;
\item Problem \eqref{PropJLE0} has no regular solutions for $\lambda>\lambda_*$.
\end{itemize}
\end{proposition}
The case $N\geq 3$ in the above result is discussed in \cite[Section X]{JL72} while the case $N=2$ follows from \cite[Theorem 2.6]{KS01}.

\begin{proposition}[{see \cite{C05}}]\label{pcon}
Let $f:[0, \infty)\times \R\to \R$ and $g:[0,\infty)\times [0, \infty)\to [0, \infty)$ be continuous functions such that:
\begin{enumerate}
\item[(i)] $|f(t,z)|\leq g(t, |z|)$ for all $z\in \R$, $t\geq 0$;
\item[(ii)] the mapping $[0, \infty)\ni z\longmapsto g(t, z)$ is nondecreasing for all $t\geq 0$.
\end{enumerate}
Then, for every $m>0$ for which
\begin{equation}\label{intcon}
\int_0^\infty g(t, 2mt)dt<m,
\end{equation}
there exists a global solution $z:[0, \infty)\to \R$ of 
\[
\begin{cases}
z''(t)+f(t, z(t))=0, \; z>0 & \mbox{ for all }\, t>0,\\
z(0)=0, &
\end{cases}
\] such that $z(t)/t\to m$ as $t\to \infty$.
\end{proposition}

\section{Proof of Theorem \ref{thmain1}}
Let $U$ be a solution of \eqref{E4'}. The proof will be divided into three cases.
\medskip

\noindent{\it Case 1: }$0<\delta <1$. 
Let $u(r):=1-(1-U(r))^{1-\delta}$ and $\tlambda:=(1-\delta)\lambda$.
Then (\ref{E4'}) reads
\begin{equation}\label{Ev1}
\begin{cases}
u''+\frac{N-1}{r}u'+\frac{\tlambda}{(1-u)^p}=0 & \textrm{for}\ 0<r<1,\\
u(0)=1-(1-\alpha)^{1-\delta},\\
u'(0)=0,\\
u(1)=0,
\end{cases}
\end{equation}
where $p:=(1+\delta)/(1-\delta)>1$.
This is equivalent to problem (\ref{E2v1}) with $\delta=0$ and has been studied by many authors.
The existent results in \cite{GM20, M18} for (\ref{Ev1}) can be summarized as follows:
\begin{itemize}
\item If $p>1$, then
\begin{equation}\label{Ev1E1}
\tlambda\to\tlambda^*:=\frac{2}{p+1}\left(N-2+\frac{2}{p+1}\right)\quad\mbox{ as }\; u(0) \to 1.
\end{equation}
\item If $p>1$, then (\ref{Ev1}) has a unique rupture solution $(\tlambda^*,1-r^{2/(p+1)})$.
\item If $p\le p_c$, then the bifurcation curve is of Type I.
\item If $p>p_c$, then the bifurcation curve is of Type II.
\end{itemize}
Here,
\[
p_c:=
\begin{cases}
\infty & \textrm{if}\ N\ge 10,\\
-1+\frac{4}{4-N+2\sqrt{N-1}} & \textrm{if}\ 2\le N<10.
\end{cases}
\]
The bifurcation was discussed in \cite[Theorem 1.2]{M18} while the uniqueness of the rupture solution and the convergence (\ref{Ev1E1}) was recently obtained in \cite{GM20}.
Note that $\lambda=\frac{\tlambda}{1-\delta}\to N-1-\delta=\lambda^*$.
Since $p=(1+\delta)/(1-\delta)$, we see that if $(N-2-2\sqrt{N-1})/2<\delta<1$ (resp. $\delta\le (N-2-2\sqrt{N-1})/2$), then $p>p_c$ (resp. $p\le p_c$).
The assertions (i) and (ii) follow from the above bifurcation result for (\ref{Ev1}).

\medskip

\noindent{\it Case 2:} $\delta =1$. 
Let $u(r):=-2\log(1-U(r))$ and $\tlambda=2\lambda$. Then, $u$ satisfies 
\begin{equation}\label{Ev3}
\begin{cases}
u''+\frac{N-1}{r}u'+\tlambda e^u=0 & \textrm{for}\ 0<r<1,\\
u(0)=-2\log (1-\alpha)>0,\\
u'(0)=0,\\
u(1)=0.
\end{cases}
\end{equation}

Classical results \cite{JL72, MP88, MN20} for (\ref{Ev3}) yield:
\begin{itemize}
\item If $N>2$, then  $\tlambda\to\tlambda^*:=2(N-2)$ as $u(0)\to\infty$.
\item If $N>2$, then (\ref{Ev3}) has a unique singular solution $(\tlambda^*,-2\log r)$.
\item If $2<N<10$, then the bifurcation curve $\tlambda(u(0))$ oscillates around $\tlambda^*$ as $u(0)\to\infty$.
\item If $N\ge 10$, then the bifurcation curve $\tlambda(u(0))$ is strictly increasing in $u(0)$.
\end{itemize}
We refer the reader to  {\cite[Section IX]{JL72}} for the bifurcation result and the convergence, to \cite{MP88} for the uniqueness of the singular solution as well as to \cite{MN20} for the uniqueness and convergence.
Note that $\lambda=\frac{\tlambda}{2}\to N-2$.
The assertions (i) and (ii) in Theorem~\ref{thmain1} follow from the above results on \eqref{Ev3}.

\medskip

\noindent{\it Case 3:} $1<\delta <N/2$, $N\geq 3$. 
Let $u(r):=(1-U(r))^{-(\delta-1)}-1$ and $\tlambda=(\delta-1)\lambda$.
Then (\ref{E4'}) becomes
\begin{equation}\label{Ev2}
\begin{cases}
u''+\frac{N-1}{r}u'+\tlambda(u+1)^p=0 & \textrm{for}\ 0<r<1,\\
u(0)=(1-\alpha)^{-(\delta-1)}-1>0,\\
u'(0)=0,\\
u(1)=0,
\end{cases}
\end{equation}
where $p=(\delta+1)/(\delta-1)>1$.
Classical results \cite{JL72,MN18,SZ95} related to (\ref{Ev2}) show that:
\begin{itemize}
\item If $p>p_{\rm S}$, then 
$$
\tlambda\to\tlambda^*:=\frac{2}{p-1}\left(N-2-\frac{2}{p-1}\right)\quad\mbox{ as }\; u(0)\to 1.
$$
\item If $p>p_{\rm S}$, then $(\tlambda^*, r^{2/(p-1)}-1)$ is the unique singular solution of (\ref{Ev2}).
\item If $p_{\rm S}<p<p_{\rm JL}$, then the bifurcation curve $\tlambda(u(0))$ oscillates around $\tlambda^*$ as $u(0)\to\infty$.
\item If $p\ge p_{\rm JL}$, then, the bifurcation curve $\tlambda(u(0))$ is strictly increasing in $u(0)$.
\end{itemize}
Here,
\[
p_{\rm S}:=
\begin{cases}
\infty & \textrm{if}\ N\le 2,\\
\frac{N+2}{N-2} & \textrm{if}\ N>2,
\end{cases}
\qquad
p_{\rm JL}:=
\begin{cases}
\infty & \textrm{if}\ N<11,\\
1+\frac{4}{N-4-2\sqrt{N-1}} & \textrm{if}\ N\ge 11.
\end{cases}
\]
The above bifurcation and the convergence results are discussed in \cite[Section X]{JL72}; the uniqueness of the singular solutions follows from \cite{SZ95}. The uniqueness and the convergence is also obtained in  \cite{MN18}.
Note that $\lambda=\frac{\tlambda}{\delta-1}\to N-1-\delta=\lambda^*$.
Since $p=(\delta+1)/(\delta-1)$, we see that if $(N-2-2\sqrt{N-1})/2<\delta<N/2$ (resp. $1<\delta\le(N-2-2\sqrt{N-1})/2$), then $p_{\rm S}<p<p_{\rm JL}$ (resp. $p\ge p_{\rm JL}$).
The assertions (i) and (ii) follow from the above bifurcation result for (\ref{Ev2}).

\section{Proof of Theorem \ref{thmain2}}

(i) Let $U$ be a regular solution of problem \eqref{E2} and denote by $\varphi_1$ the first eigenfunction of the Dirichlet Laplacian on $B$ such that $\varphi_1>0$ in $B$. Since $U$ is superharmonic and nonconstant we have $U>0$ in $B$. From this one easily obtains that $\frac{1}{1-U}\ge 4U$ in $B$, so $U$ satisfies 
\[
-\Delta U\geq 4\lambda U\quad\mbox{ in }\;B,
\]
and strict inequality holds on a set with positive Lebesgue measure.
We multiply by $\varphi_1$ in the above inequality and integrate over $B$. We find 
\[
\mu_1\int_B\varphi_1Udx=-\int_B\Delta\varphi_1Udx
=-\int_B\varphi_1\Delta Udx>4\lambda\int_B \varphi_1 U dx,
\]
which yields $\mu_1>4\lambda$. Let us now establish the inequality $\lambda<\mu_1/\delta$. Since $0<U<1$ in $B$ we have 
\begin{equation}\label{Uineq}
-\Delta U>\lambda+\delta |\nabla U|^2\quad\mbox{ in }\;B.
\end{equation}
Let $V=e^{\delta U}-1$ which from \eqref{Uineq} and $U=0$ on $\partial B$ satisfies 
\[
-\Delta V\geq \lambda\delta (1+V) \quad\mbox{ in }\;B
\]
and $V=0$ on $\partial B$. We next multiply the above inequality by $\varphi_1$, integrate over $B$ and  proceed as before to deduce $\lambda\delta<\mu_1$.

\medskip

\noindent{\bf Remark.}  The above method works also for problem \eqref{E2v1}. 
The relevant inequalities that one employs are $\frac{1}{(1-U)^p}\geq 4U$ and $\frac{1}{(1-U)^p}>1$ in $B$, which  hold for all $p\geq 1$.
\medskip

(ii) Let $u(r):=(1-U(r))^{-(\delta-1)}-1$, $\tlambda=(\delta-1)\lambda$ and $p:=(\delta+1)/(\delta-1)$.
Then $(\tlambda,u)$ is a classical solution of
\begin{equation}\label{L9E1}
\begin{cases}
\Delta u+\tlambda(u+1)^p=0 & \textrm{in}\ B,\\
u=0 & \textrm{on}\ \partial B,
\end{cases}
\end{equation}
if and only if $(\lambda,U)$ is a regular solution of $(\ref{E2})$.
Since $\delta\ge N/2$, we see that $1<p\le (N+2)/(N-2)$.
By Proposition~\ref{PropJL} we see that the first assertion of (ii) holds.

It follows from \cite[Theorem 12]{GM02} that if $\tlambda\le 2N(p-1)^{p-1}/p^p$, then (\ref{L9E1}) has a classical solution.
Since $p=(\delta+1)/(\delta-1)$, we see that if
\[
0<\lambda\le\frac{1}{\delta-1}2N\frac{(p-1)^{p-1}}{p^p}=N\left(\frac{2}{\delta+1}\right)^{(\delta+1)/(\delta-1)},
\]
then (\ref{E2}) has a solution.
By Proposition~\ref{PropJL} we see that \eqref{barl} holds. 

(iii) Observe that the family (\ref{L9E0}) satisfies (\ref{E2}). The conclusion follows by part (ii) above.

\section{Proof of Theorem \ref{thmain3}}

(i) Let us first assume $\delta =1$ and prove that \eqref{E2} has infinitely many solutions for any $\lambda\in (0,1)$. Let $v(r)=\log 2\lambda-2\log(1-U(r))$.
Then (\ref{E2}) becomes
\begin{equation}\label{S7E4}
\begin{cases}
\Delta v+e^v=0 & \textrm{in}\ B,\\
v=0 & \textrm{on}\ \partial B.
\end{cases}
\end{equation}
By \cite[Theorem 1.1 (ii)]{T06} problem (\ref{S7E4}) has a two-parameter family of singular solutions
\[
v(r)=\log a+(b-2)\log r-2\log\left(1+\frac{a}{2b^2}r^b\right)\ \ \textrm{for $a>0$ and $0<b<2$.}
\]
Since $U(1)=0$, we obtain
\[
\lambda=\frac{2ab^4}{(a+2b^2)^2}\ \ \textrm{and}\ \ 
U(r)=1-\frac{ar^b+2b^2}{a+2b^2}r^{(2-b)/2}\ \ \textrm{for $a>0$ and $0<b<2$.}
\]
We easily see that for each $\lambda\in (0,1)$ the above solutions consist of a one-parameter family of rupture solutions of (\ref{E2}).
Thus, the assertion (i) holds.

Assume now $\delta>1$.
It is easy to see that if $U$ is a solution of \eqref{E2}, then $v(r)=(1-U(r))^{1-\delta}$ satisfies
\begin{equation}\label{L13E1}
\begin{cases}
v''+\frac{N-1}{r}v'+\lambda(\delta -1) v^{p}=0 & \textrm{ for }\ 0<r<1,\\
v(1)=1, & 
\end{cases}
\end{equation}
where $p=(\delta+1)/(\delta -1)>1$. Further, let $t=-\ln r\in [0, \infty)$ and $w(t)=v(r)$. From \eqref{L13E1} we obtain that $w$ satisfies
\begin{equation}\label{L13E2}
\begin{cases}
w_{tt}+\lambda(\delta -1)e^{-2t}w^p=0  & \textrm{ for }\ 0<t<\infty,\\
w(0)=1. & 
\end{cases}
\end{equation}

Let $z(t):=w(t)-1$.
Then (\ref{L13E2}) becomes
\begin{equation}\label{L13E3}
\begin{cases}
z_{tt}+f(t,z)=0 & \textrm{for}\ 0<t<\infty,\\
z(0)=0,
\end{cases}
\end{equation}
where $f(t,z):=\lambda (\delta-1)e^{-2t}|z+1|^p$.
We have
\[
f(t,z)\le 2^{p-1}\lambda(\delta-1)e^{-2t}+2^{p-1}\lambda(\delta-1)e^{-2t}|z|^p=:g(t,z).
\]
Then, $g(t,z)$ is increasing in $z>0$. In order to apply Proposition \ref{pcon} we need to check that condition \eqref{intcon} is fulfilled for small $m>0$. 
Indeed, we have
\[
\int_0^{\infty}g(t,2mt)dt
=2^{p-2}\lambda(\delta-1)+2^{2p-1}\lambda(\delta-1)m^pa,
\]
where $a:=\int_0^{\infty}t^pe^{-2t}ds$.
For each small $\lambda>0$, there is an interval $I\subset(0,\infty)$ such that $\int_0^{\infty}g(t,2mt)dt<m$ for all $m\in I$, since
\[
\frac{2^{p-2}\lambda(\delta-1)}{m}+2^{2p-1}\lambda (\delta-1)am^{p-1}<1
\ \ \textrm{for}\ \ m\in I.
\]
By Proposition \ref{pcon}, problem (\ref{L13E3}) has a solution $z(t)$ such that $z>0$ on $(0, \infty)$ and $z(t)/t\to m$ as $t\to\infty$. Hence, $w=z+1\geq 1$ is a solution of (\ref{L13E2}) and thus, for each small $\lambda>0$, (\ref{L13E1}) has infinitely many solutions.\\

(ii) Let $v(r):=(\delta-1)^{(\delta-1)/2}(1-U(r))^{-(\delta-1)}$.
Then
\begin{equation}\label{L7E0}
\begin{cases}
v''+\frac{N-1}{r}v'+\lambda v^p=0 & \textrm{for}\ 0<r<1,\\
v(1)=(\delta-1)^{(\delta-1)/2},\\
\end{cases}
\end{equation}
where $p:=(\delta+1)/(\delta-1)$.
Let
\[
x(t):=\frac{v(r)}{\{(\delta-1)(N-1-\delta)\lambda^{-1}\}^{(\delta-1)/2}r^{-\delta+1}}
\quad\textrm{and}\quad t:=-\log r.
\]
Then $x(t)$ satisfies
\begin{equation}\label{L7E1}
\begin{cases}
x'=y,\\
y'=(N-2\delta)y+(\delta-1)(N-1-\delta)(x-x^p)
\end{cases}
\end{equation}
for $t>0$.
Suppose that $(x(t),y(t))$ satisfies (\ref{L7E1}).
Let
\[
E(x,y):=\frac{1}{2}y^2{-}(\delta-1)(N-1-\delta)\left(\frac{x^2}{2}-\frac{x^{p+1}}{p+1}\right).
\]
Since $N/2\le\delta<N-1$, we have
\begin{equation}\label{L7E2}
\frac{d}{dt}E(x(t),y(t))=y\left\{y'-(\delta-1)(N-1-\delta)(x-x^p)\right\}=-(2\delta-N)y^2\le 0
\end{equation}
and hence $E$ is a Lyapunov function.
It is well known that $\Omega:=\{(x,y);\ x>0,\ E(x,y)<0\}$ is a tear-shaped region, and it follows from (\ref{L7E2}) that if $(x(t_0),y(t_0))\in\Omega$, then $(x(t),y(t))\in\Omega$ for all $t\ge t_0$.
If $\delta=N/2$, then the orbit $\{(x(t),y(t))\}$ is periodic, and a simple calculation shows that there is $c_0>0$ such that $c_0\le x(t)<x_{\delta}$, where $x_{\delta}:=\{\delta/(\delta-1)\}^{(\delta-1)/2}$.
Then $v(r)\to\infty$ as $r\to 0$.
The corresponding solution $v(r)$ is a singular solution, and hence the corresponding solution $U(r)$ is a rupture solution.
When $N/2<\delta<N-1$, it follows from (\ref{L7E2}) that there is no periodic orbit.
Since $\Omega$ is bounded, by Poincar\'{e}-Bendixon theorem the orbit $\{(x(t),y(t))\}$ converges to an equilibrium point in $\Omega$, which is $(1,0)$.
Therefore, $x(t)\to 1$ as $t\to\infty$.
The corresponding solution $v(r)$ is a singular solution, and the corresponding solution $U(r)$ is a rupture solution.

It is enough to show that for each $\lambda\in(0,\lambda_0)$, there are infinitely many initial data $(x(0),y(0))\in\Omega$ such that
\begin{equation}\label{S7E3}
v(1)=\{(\delta-1)(N-1-\delta)\lambda^{-1}\}^{(\delta-1)/2}x(0)=(\delta-1)^{(\delta-1)/2}.
\end{equation}
Here, the boundary condition in (\ref{L7E0}) is satisfied if (\ref{S7E3}) holds.
Let
$$
y_{\delta,x}:=\sqrt{2(\delta-1)(N-1-\delta)\left(\frac{x^2}{2}-\frac{x^{p+1}}{p+1}\right)}.
$$
Now, let $\lambda\in (0,\lambda^{***})$ be fixed, where $\lambda^{***}>0$ is given in \eqref{l3} and let $x(0)=\left(\frac{\lambda}{N-1-\delta}\right)^{(\delta-1)/2}$.
Then, $0<x(0)<x_{\delta}$ and (\ref{S7E3}) holds.
Let $|y(0)|<y_{\delta,x(0)}$.
Then $(x(0),y(0))\in\Omega$.
Hence, the orbit $\{(x(t),y(t))\}$ with initial data $(x(0),y(0))$ is in $\Omega$ for $t>0$.
Moreover, the orbit is uniformly away from $(0,0)$, since $E(x(t),y(t))\le E(x(0),y(0))$ for $t>0$.
Hence the corresponding solution $U$ of \eqref{E2} is a rupture solution.

The asymptotic expansion (\ref{L7E-2}) follows from the fact that $x(t)\to 1$ ($t\to\infty$), and (\ref{L7E-1}) follows from the fact that $c_0\le x(t)< x_{\delta}$. 

(iii)  Letting $v=(1-U)^{1-\delta}$, we see that $v$ satisfies \eqref{L13E1} 
where $p=(\delta+1)/(\delta -1)$ fulfills $1<p\leq N/(N-2)$.
The rest of the proof follows from the result below.

\begin{lemma}\label{L10}
Let $N\ge 3$, $1<p\leq N/(N-2)$ and $a> 0$. Then, there exists $\Lambda_0>0$ such that for all $\lambda\in (0, \Lambda_0)$ the problem 
\begin{equation}\label{L10E1}
\begin{cases}
v''+\frac{N-1}{r}v'+\lambda v^{p}=0 & \textrm{ for }\ 0<r<1,\\
v(1)=a, & 
\end{cases}
\end{equation}
has infinitely many singular solutions.  
\end{lemma}
\begin{proof}  We shall analyse separately the cases $p=N/(N-2)$ and $1<p<N/(N-2)$ for which we provide different arguments.
\smallskip

\noindent{\it Case 1:} $p=N/(N-2)$. By the proof of Theorem~1(i) in \cite{NS85}, there exists $\alpha^*>0$ such that 
for any $\alpha\in (0, \alpha^*)$ the problem
\begin{equation}\label{L10Ea}
\begin{cases}
v''+\frac{N-1}{r}v'+v^{N/(N-2)}=0 & \textrm{ for }\ 0<r<1,\\
v(1)=0\,,v'(1)=-\alpha, & 
\end{cases}
\end{equation}
has a singular solution $v_\alpha$. (More precisely,
$\alpha^*=(N-2)\alpha_0$, where $\alpha_0$ appears in the proof in \cite[Theorem 1(i), page 664]{NS85}.) By the result in \cite[Theorem A]{A87} we have 
\begin{equation}\label{avileslim}
\lim_{r\to 0}\frac{v_\alpha(r)}{r^{2-N}\Big(\log\frac{1}{r}\Big)^{-\frac{N-2}{2}}}=\Big(\frac{N-2}{\sqrt{2}}\Big)^{N-2}.
\end{equation}
Fix $\alpha\in (0,\alpha^*)$ and let $a>0$. Take $\lambda>0$ small such that 
\[
\max_{0\leq \rho\leq 1}\rho^{N-2}v_\alpha(\rho)>\lambda^{\frac{N-2}{2}}a.
\]
Note that by the asymptotic behavior of $v_\alpha$ given by \eqref{avileslim}, the above maximum is finite and is achieved inside of the interval $(0,1)$. 
By the continuous dependence of solutions $v_\alpha$ on $\alpha$, one can find a small interval $I\subset (0, \alpha^*)$ centred at $\alpha$ such that 
\[
\max_{0\leq \rho\leq 1}\rho^{N-2}v_\beta(\rho)>\lambda^{\frac{N-2}{2}}a,
\]
for any $\beta\in I$. Thus, for any $\beta\in I$ there exists $\rho_\beta\in (0,1)$ such that 
\[
\rho_\beta^{N-2}v_{\beta}(\rho_\beta)=\lambda^{\frac{N-2}{2}}a.
\]
Then, the family $\{V_\beta\}_{\beta\in I}$ given by
\[
V_\beta(r)=\rho^{N-2}_\beta \lambda^{-\frac{N-2}{2}}v_\beta(\rho_\beta r)\,\;\; \mbox{ for } 0<r\leq 1,
\]
consists of singular solutions of \eqref{L10E1}. Let us note that if 
$\beta_1, \beta_2\in I$, $\beta_1\neq \beta_2$ then $V_{\beta_1}\neq V_{\beta_2}$. Indeed, if this was not true, by the local uniqueness of solutions to regular ODE,  $V_{\beta_1}(r)=V_{\beta_2}(r)$ for all $0<r\leq \min\{1/\rho_{\beta_1}, 1/\rho_{\beta_2}\}$ which contradicts the definition of $v_{\beta_1}$ and $v_{\beta_2}$ as solutions of \eqref{L10Ea} with $\alpha=\beta_1$ and $\alpha=\beta_2$ respectively. Thus, we provided an infinite family of singular solutions to \eqref{L10E1} which proves Theorem~\ref{thmain3}(iii) in the case $p=N/(N-2)$.

\smallskip

\noindent{\it Case 2:} $1<p<N/(N-2)$.
It is proved in \cite{GV88} that if $a=0$ then \eqref{L10E1} has infinitely many solutions for any $\lambda>0$.  For $a>0$ we provide a different approach based on Proposition \ref{pcon}. Let $a>0$ and $w(t)=v(r)$, $t=r^{2-N}$. Then $w$ satisfies
\begin{equation}\label{L7E03a}
\begin{cases}
w_{tt}+\frac{\lambda}{(N-2)^2}t^{-2(N-1)/(N-2)} w^p=0 & \textrm{for}\ t>1,\\
w(1)=a.
\end{cases}
\end{equation}
Consider the problem
\begin{equation}\label{L7E03b}
\begin{cases}
z_{tt}+\frac{\lambda}{(N-2)^2}(t+1)^{-2(N-1)/(N-2)} (z+a)^p=0 & \textrm{for}\ t>0,\\
z(0)=0.
\end{cases}
\end{equation}
It is not hard to see that for any $1<p<N/(N-2)$, the function $g(t,z)=\frac{\lambda}{(N-2)^2}(t+1)^{-2(N-1)/(N-2)} (z+a)^p$ satisfies condition \eqref{intcon} provided $\lambda>0$ is small enough. By Proposition \ref{pcon} it follows that \eqref{L7E03b} has infinitely many positive solutions
with $z(t)\to \infty$ as $t\to \infty$. Letting now $w(t)=z(t-1)+a$, it follows that \eqref{L7E03a} has infinitely many solutions with $w(t)\to \infty$ as $t\to \infty$ provided $\lambda>0$ is small. This implies in turn that for any small $\lambda>0$, problem \eqref{L10E1} has infinitely many singular solutions and the proof of  Lemma~\ref{L10} is now complete.
\end{proof}

\end{document}